\documentclass[12pt,a4paper]{article}
\usepackage[cp1251]{inputenc}
\usepackage{graphicx}
\usepackage[hidelinks]{hyperref}
\usepackage[english,russian]{babel}
\usepackage{amsmath,amsthm,dsfont,pifont,amsfonts,MnSymbol}
\usepackage{MPatr_Style_English_09.2014}

\DeclareMathOperator{\lh}{\mathsf{lh}}
\DeclareMathOperator{\length}{\mathsf{length}}

\DeclareMathOperator{\inn}{\mathsf{in}}
\DeclareMathOperator{\outt}{\mathsf{out}}

\author{Mikhail Patrakeev
}

\title{A simple closed curve in $\mathbb{R}^3$ whose convex hull equals\\
the half-sum of the curve with itself
}

\date{}

\begin{document}
\hyphenation{no-n-in-cre-a-s-ing tran-si-ti-ve par-ti-al}
\renewcommand{\proofname}{\textup{\textbf{Proof}}}
\renewcommand{\abstractname}{\textup{Abstract}}
\renewcommand{\refname}{\textup{References}}
\maketitle
\begin{abstract}
If $\Gamma$ is the range of a Jordan curve that bounds a convex set in $\mathbb{R}^2,$ then $\frac{1}{2}(\Gamma+\Gamma)=\mathsf{co}(\Gamma),$ where $+$ is the Minkowski sum and $\mathsf{co}$ is the convex hull.
Answering a question of V.N.\,Ushakov, we construct a simple closed curve in $\mathbb{R}^3$ with range $\Gamma$ such that $\frac{1}{2}(\Gamma+\Gamma)=[0,1]^3=\mathsf{co}(\Gamma).$
Also we show that such simple closed curve cannot be rectifiable.
\end{abstract}

\begin{center}
\includegraphics[scale=0.60]{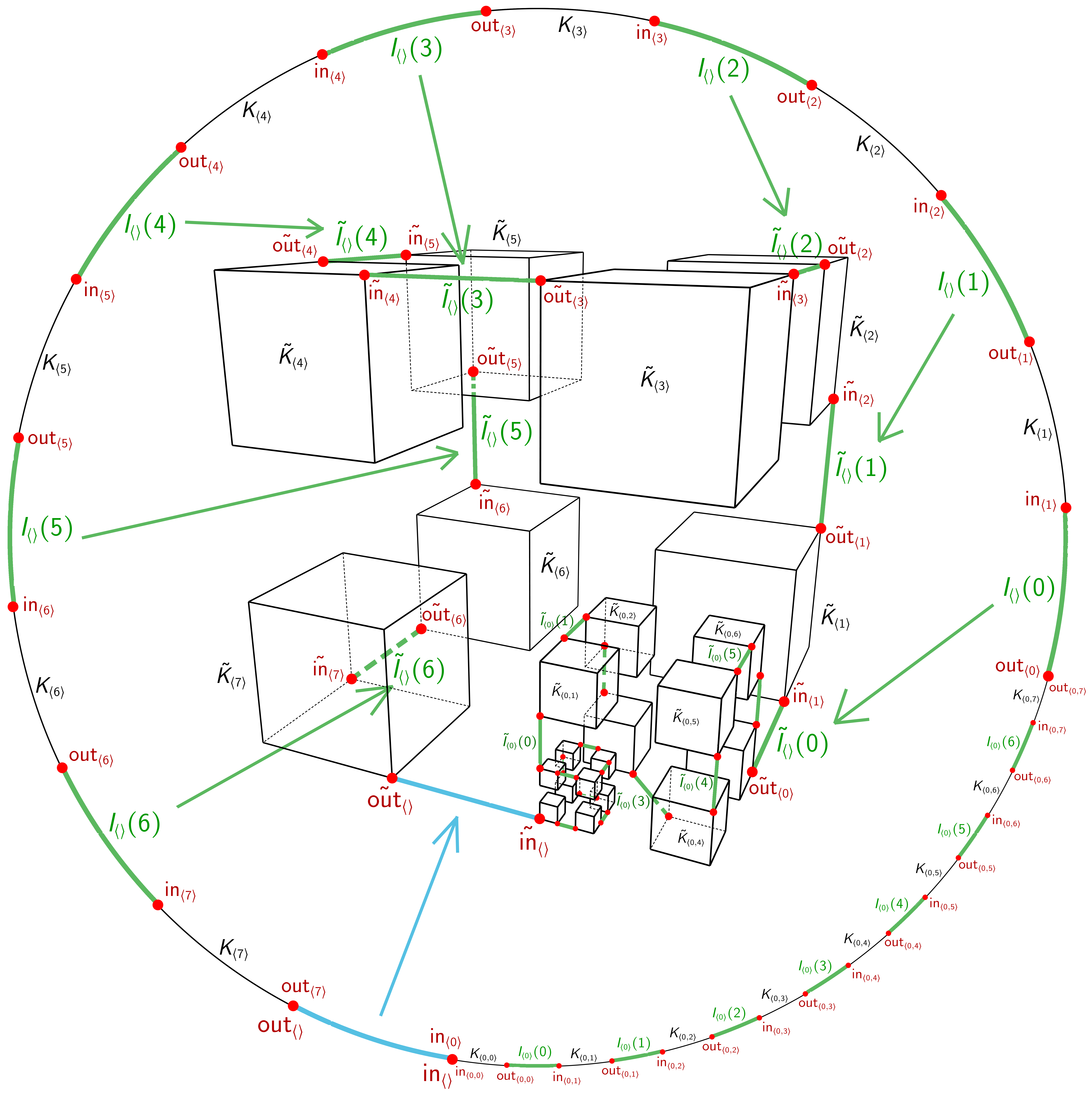}
\end{center}

\section{Introduction}
\label{section0}

Recall that the Minkowski sum ${A}+{B}$ of two sets in the $\mathsurround=0pt{n}$-dimensional Euclidean space $\mathbb{R}^{n}$ is the set $\{{x}+{y}:{x}\in{A},{y}\in{B}\};$ recall also that ${r}{A}\coloneq\{{r}{x}:{x}\in{A}\}$ for ${r}\in\mathbb{R}.$
It follows from the Shapley---Folkman theorem~\cite{starr} that the sequence
$$
\textstyle\Gamma,\ \frac{1}{2}(\Gamma+\Gamma),\ \frac{1}{3}(\Gamma+\Gamma+\Gamma),\ \frac{1}{4}(\Gamma+\Gamma+\Gamma+\Gamma), \ \ldots
$$
converges in the Hausdorff metric to the convex hull $\mathsf{co}(\Gamma)$ of $\Gamma$ for any bounded set $\Gamma\subseteq\mathbb{R}^{n}.$

If $\Gamma$ is the range of a simple closed curve that bounds a convex set in $\mathbb{R}^2,$  then already the second term in this sequence reaches the limit, so that $\frac{1}{2}(\Gamma+\Gamma)=\mathsf{co}(\Gamma)$ for such $\Gamma.$
Recently V.N.\,Ushakov studied the rate of convergence of the above sequence~\cite{ushakov}  and he asked the following question:

\begin{que}\label{q}
\textit{Is there a simple closed curve in $\mathbb{R}^3$ whose range $\Gamma$ does not lie in any plane and such that $\frac{1}{2}(\Gamma+\Gamma)=\mathsf{co}(\Gamma)$?}
\end{que}

Note that $\Gamma$ does not lie in any plane if and only if $\mathsf{co}(\Gamma)$ contains a tetrahedron.
Since $\Gamma$ in this question is a homeomorphic image of the unit circle $\mathbb{S}^1,$ it follows that $\Gamma+\Gamma$ is a continuous image of the torus $\mathbb{S}^1\!\times\mathbb{S}^1.$ Therefore in smooth cases $\frac{1}{2}(\Gamma+\Gamma)$ looks like a two-dimensional manifold with self-intersections, so it should not contain a tetrahedron. Indeed, we show that even in the rectifiable case the answer to the above question is negative:

\begin{teo}\label{rem}
   If $\hspace{1pt}\Gamma$ is the range of a rectifiable curve in $\hspace{1pt}\mathbb{R}^3,$
   then $\hspace{1pt}\frac{1}{2}(\Gamma+\Gamma)$ has zero Lebesgue measure in $\hspace{1pt}\mathbb{R}^3.$
\end{teo}

\noindent On the other hand, we build an example which shows that in the general case the answer is positive:

\begin{teo}\label{th}
   There is a simple closed curve in $\hspace{1pt}\mathbb{R}^3$ whose range $\hspace{1pt}\Gamma$ satisfies
   $\hspace{1pt}\frac{1}{2}(\Gamma+\Gamma)=[0,1]^3=\mathsf{co}(\Gamma).$
\end{teo}

\section{Proofs}
\label{section1}

\begin{proof}[\textbf{\textup{Proof of Theorem\,\ref{rem}}}]
We denote by $\omega$ the set of natural numbers, by $\mathsf{diam}(A)$ the diameter of a set ${A},$ and by ${f}{\upharpoonright}\hspace{0.7pt}{A}$ the restriction of a function $f$ to a set ${A}.$
Suppose that $\gamma\colon[a,b]\to\mathbb{R}^3$ is a rectifiable curve~\cite{rudin} and $\Gamma=\mathsf{range}(\gamma).$
Let $\varepsilon>0.$ Fist we build a sequence $\langle{t}_{{i}}\rangle_{{i}\in\omega}$ of points in the segment $[{a},{b}]$ by recursion on ${i}\in\omega$:\\
\mbox{}\ \ Base:\ \
   ${t}_{0}\coloneq{a}.$\\
\mbox{}\ \ Step:\ \
   ${t}_{{i}+1}\coloneq
   \mathsf{sup}\big\{{u}\in[{t}_{{i}},{b}]:
   \mathsf{diam}\big(\gamma([{t}_{{i}},{u}])\big)<\varepsilon\big\}\quad$
   for all ${i}\in\omega.$\\
Note that if ${t}_{{i}}<{b},$ then ${t}_{{i}+1}>{t}_{{i}}$ because $\gamma$ is continuous.
Note also that
$\mathsf{diam}\big(\gamma([{t}_{{i}},{t}_{{i}+1}])\big)
=\varepsilon$ for all ${i}\in\omega$ such that ${t}_{{i}+1}<{b},$
and then $\mathsf{length}(\gamma{\upharpoonright}\hspace{0.7pt}[{t}_{{i}},{t}_{{i}+1}])\geqslant\varepsilon$ for all such ${i}.$ Since $\gamma$ is a rectifiable curve, the last inequality implies that there is ${i}\in\omega$ such that ${t}_{{i}+1}={b};$ let ${n}$ be the first such ${i}.$
Put ${L}_{{i}}\coloneq
\gamma([{t}_{{i}},{t}_{{i}+1}])$ for all ${i}\leqslant{n}.$

We have $\mathsf{diam}({L}_{{i}})=\varepsilon$ for all ${{i}}<{n}$ and
$\mathsf{diam}({L}_{{n}})\leqslant\varepsilon.$
It follows that $\mathsf{diam}({L}_{{i}}+{L}_{{j}})
\leqslant2\varepsilon$ for all ${i},{j}\leqslant{n}$ because $\mathsf{diam}({A}+{B})\leqslant\mathsf{diam}({A})+\mathsf{diam}({B})$
for any ${A},{B}\subseteq\mathbb{R}^{3}$ (indeed, if ${a}_{1},{a}_{2}\in{A}$ and ${b}_{1},{b}_{2}\in{B},$ then $\rho({a}_{1}+{b}_{1},{a}_{2}+{b}_{2})\leqslant
\rho({a}_{1}+{b}_{1},{a}_{1}+{b}_{2})+\rho({a}_{1}+{b}_{2},{a}_{2}+{b}_{2})=
\rho({b}_{1},{b}_{2})+\rho({a}_{1},{a}_{2})\leqslant
\mathsf{diam}({B})+\mathsf{diam}({A})$).
Then $\mu_{3}^\ast({L}_{{i}}+{L}_{{j}})
<(2\cdot2\varepsilon)^3=64\varepsilon^3$ for all ${{i},{j}\leqslant{n}},$
where $\mu_{3}^\ast$ is the Lebesgue outer measure in $\mathbb{R}^3.$

Now we have
$$\length(\gamma)\:\geqslant
\sum_{0\leqslant{i}\leqslant{n}}\mathsf{diam}({L}_{{i}})\:\geqslant
\sum_{0\leqslant{i}<{n}}\mathsf{diam}({L}_{{i}})\:=\:
\varepsilon{n}\:=\:
\varepsilon({n}+1)-\varepsilon,
$$
so
${n}+1\,\leqslant\,\varepsilon^{{-}1}\cdot\length(\gamma)+1.$
Then
\begin{multline*}
\mu_{3}^\ast\big({\textstyle\frac{1}{2}(\Gamma+\Gamma)}\big)\:\leqslant\:
\mu_{3}^\ast(\Gamma+\Gamma)\:=\:
\mu_{3}^\ast\big(\!\!\!\bigcup_{0\leqslant{i}\leqslant{n}}\!\!\!{L}_{{i}}
+\!\!\!\bigcup_{0\leqslant{j}\leqslant{n}}\!\!\!{L}_{{j}}
\big)\:=\:
\mu_{3}^\ast\big(\!\!\!\bigcup_{0\leqslant{i},{j}\leqslant{n}}\!\!\!
({L}_{{i}}+{L}_{{j}})\big)\:\leqslant\:
\sum_{0\leqslant{i},{j}\leqslant{n}}
\mu_{3}^\ast({L}_{{i}}+{L}_{{j}})\:<\\
<\:64\varepsilon^3({n}+1)^2\:\leqslant\:
64\varepsilon^3\big(\varepsilon^{{-}1}\cdot\length(\gamma)+1\big)^{2}\:=\:
64\varepsilon\big(\length(\gamma)^2+2\varepsilon\length(\gamma)+\varepsilon^2\big).
\end{multline*}
Since $\varepsilon$ can be arbitrarily small, it follows that the set $\frac{1}{2}(\Gamma+\Gamma)$ has zero Lebesgue measure in $\mathbb{R}^3.$
\end{proof}
\medskip

Let $\mathbb{C}\subseteq[0,1]$ be the Cantor set~\cite{rudin}, also known as the Cantor ternary set or as the middle third Cantor set.

\begin{lem}\label{lem1}
   If $\hspace{1pt}\hspace{1pt}\mathbb{C}^3\subseteq\Gamma\subseteq[0,1]^3,$ then $\hspace{1pt}\frac{1}{2}(\Gamma+\Gamma)=[0,1]^3=\mathsf{co}(\Gamma).$
\end{lem}

\begin{lem}\label{lem2}
   There exists a simple closed curve in $\hspace{1pt}\mathbb{R}^3$ whose range $\hspace{1pt}\Gamma$ satisfies
   $\hspace{1pt}\mathbb{C}^3\subseteq\Gamma\subseteq[0,1]^3.$
\end{lem}

\begin{proof}[\textbf{\textup{Proof of Theorem\,\ref{th}}}]
Theorem\,\ref{th} immediately follows from the above lemmas.
\end{proof}

\begin{proof}[\textbf{\textup{Proof of Lemma\,\ref{lem1}}}]
Suppose that $\hspace{1pt}\mathbb{C}^3\subseteq\Gamma\subseteq[0,1]^3.$
It is well known~\cite{rudin} that the Cantor set $\mathbb{C}$ equals the set of those numbers in $[0,1]$ that have only 0's and 2's in their ternary expansion; that is,
$\mathbb{C}=\big\{\sum_{{n}\in\omega}({a}_{n}\cdot3^{-{n}-1}):
\langle{a}_{n}\rangle_{{n}\in\omega}\in{}^\omega\{0,2\}\big\},$
where
${}^\omega\!{A}$ is the set of infinite sequences in ${A}.$ Also it is known that $\mathbb{C}+\mathbb{C}=[0,2];$
indeed,
$${\textstyle\frac{\mathbb{C}}{2}}=
\big\{\sum_{{n}\in\omega}{(\textstyle\frac{{a}_{n}}{2}\cdot3^{-{n}-1})}:
\langle{a}_{n}\rangle_{{n}\in\omega}\in{}^\omega\{0,2\}\big\}=
\big\{\sum_{{n}\in\omega}({b}_{n}\cdot3^{-{n}-1}):
\langle{b}_{n}\rangle_{{n}\in\omega}\in{}^\omega\{0,1\}\big\},\quad\text{so}$$
$$
{\textstyle\frac{\mathbb{C}}{2}+{\frac{\mathbb{C}}{2}}}=
\big\{\sum_{{n}\in\omega}\big(({b}_{n}{+}\hspace{1pt}{c}_{n})\cdot3^{-{n}-1}\big):
\langle{b}_{n}\rangle_{{n}\in\omega},
\langle{c}_{n}\rangle_{{n}\in\omega}\in{}^\omega\{0,1\}\big\}=
\big\{\sum_{{n}\in\omega}({d}_{n}\cdot3^{-{n}-1}):
\langle{d}_{n}\rangle_{{n}\in\omega}\in{}^\omega\{0,1,2\}\big\}=[0,1],
$$
and hence $\mathbb{C}+\mathbb{C}=
2\big(\frac{\mathbb{C}}{2}+\frac{\mathbb{C}}{2}\big)=
2[0,1]=[0,2].$

Recall that $\frac{1}{2}({A}+{A})\subseteq\mathsf{co}({A})$ for any ${A}\subseteq\mathbb{R}^{3}$ and
$({B}+{B})^3={B}^3+{B}^3$ for any ${B}\subseteq\mathbb{R}.$
This is because
$\frac{1}{2}({A}+{A})=\big\{\frac{{x}+{y}}{2}:{x},{y}\in{A}\big\}\subseteq
\bigcup\{[{x},{y}]:{x},{y}\in{A}\}\subseteq\mathsf{co}({A})$ and
$({B}+{B})^3
=\{\langle{x}_{1}+{y}_{1},{x}_{2}+{y}_{2},{x}_{3}+{y}_{3}\rangle:{x}_{{i}},{y}_{{i}}\in{B}\}=
\{\langle{x}_{1},{x}_{2},{x}_{3}\rangle+
\langle{y}_{1},{y}_{2},{y}_{3}\rangle:{x}_{{i}},{y}_{{i}}\in{B}\}=
{B}^3+{B}^3$.

Now, using all above, we have
$$\textstyle[0,1]^3=\frac{1}{2}[0,2]^3=
\frac{1}{2}(\mathbb{C}+\mathbb{C})^3=
\frac{1}{2}(\mathbb{C}^3+\mathbb{C}^3)\subseteq
\frac{1}{2}(\Gamma+\Gamma)\subseteq
\mathsf{co}(\Gamma)\subseteq\mathsf{co}([0,1]^3)=[0,1]^3.
$$%
\end{proof}

We shall use the following notation in the proof of Lemma\,\ref{lem2}:

   \begin{itemize}
   \item [\ding{46}\,]
      $\mathsurround=0pt
      {A}\,=\,\bigsqcup_{{\lambda}\in{\Lambda}}{B}_{\lambda}
      \quad$ means
      $\quad{A}\,=\,\bigcup_{{\lambda}\in{\Lambda}}{B}_{\lambda}\enskip\mathsf{and}\enskip
      \forall{\lambda}\,{\neq}\,{\lambda}'\,{\in}\,{\Lambda}
      \;[\,{B}_{\lambda}\cap{B}_{{\lambda}'}=\varnothing\,];$
   \item [\ding{46}\,]
      $\mathsurround=0pt
      {A}\,=\,{B}_{0}\sqcup\ldots\sqcup {B}_{n}\quad$ means
      $\quad{A}\,=\,\bigsqcup_{\,{i}\in\{{0},\ldots,{n}\}}{B}_{i}\,;$
   \item [\ding{46}\,]
      $\mathsurround=0pt
      \omega$ $\coloneq$ the set of finite ordinals $=$ the set of natural numbers,

      so $0=\varnothing\in\omega$ and ${n}=\{0,\ldots,{n}-1\}$ for all ${n}\in\omega;$
   \item [\ding{46}\,]
      $\mathsurround=0pt
      {f}{\upharpoonright}\hspace{0.7pt}{A}$ $\coloneq$ the restriction of function $f$ to ${A};$
   \item [\ding{46}\,]
      $\mathsurround=0pt
      {}^{\alpha}\!{A}$ $\coloneq$ the set of functions from $\alpha$ to ${A};$
   \item [\ding{46}\,]
      $\mathsurround=0pt
      {}^{{<}\omega}\hspace{-1pt}{A}
      \,\coloneq\,\bigcup_{{n}\in\omega}{}^{{n}}\hspace{-1pt}{A}.$
   \end{itemize}
We say that ${s}$ is a \textit{sequence} iff ${s}$ is a function such that $\mathsf{domain}({s})\in\omega$ or $\mathsf{domain}({s})=\omega.$
Therefore ${}^{{<}\omega}\hspace{-1pt}{A}$ is the set of finite sequences in ${A}.$
By definition, the \textit{length} of a sequence ${s},$ denoted by $\lh({s}),$ is the domain of ${s}.$
We use the following notation when work with sequences:

   \begin{itemize}
   \item [\ding{46}\,]
      $\mathsurround=0pt
      \langle{s}_0,\ldots,{s}_{{n}-1}\rangle$ $\coloneq$
      the sequence ${s}$ of length ${n}\in\omega$ such that ${s}(i)={s}_{i}$ for all ${i}\in{n};$
   \item [\ding{46}\,]
      $\mathsurround=0pt
      \langle\rangle$ $\coloneq$ the sequence of length {0};
   \item [\ding{46}\,]
      $\mathsurround=0pt
      \langle{x}\rangle^{\alpha}$
      $\coloneq$ the sequence ${s}$ of length $\alpha$ such that ${s}_{{i}}={x}$ for all ${i}\in\alpha;$
   \item [\ding{46}\,]
      $\mathsurround=0pt
      {s}\hat{\ }{t}$ $\coloneq$ the concatenation of a finite sequence ${s}$ and a sequence ${t};$

      that is, ${s}\hat{\ }{t}$ is the sequence ${r}$ of length $\lh({s})+\lh({t})\leqslant\omega$ such that ${r}_{{i}}={s}_{{i}}$ for all ${i}\in\lh({s})$ and ${r}_{\lh({s})+{i}}={t}_{{i}}$ for all ${i}\in\lh({t});$
   \item [\ding{46}\,]
      $\mathsurround=0pt
      {s}\hspace{0.5pt}^{\frown}{x}\,\coloneq\,
      {s}\hspace{0.5pt}\hat{\ }\langle{x}\rangle,$ where ${s}$ is a finite sequence.
   \end{itemize}
Here are some examples of usage of the above notation:
$$
\lh(\langle2,5,9,7\rangle)=4=\{0,1,2,3\};\qquad
\lh(\langle\rangle)=0=\varnothing;\qquad
\lh(\langle3\rangle^\omega)=\omega;
$$
$$
\langle2,5,9,7\rangle{\upharpoonright}\hspace{0.7pt}\,{3}\,=\,
\langle2\rangle\hat{\ }\langle5,9\rangle\,=\,
\langle2,5,9\rangle\hat{\ }\langle\rangle\,=\,
\langle2,5\rangle^\frown{9}\,=\,
\langle2,5,9\rangle;\qquad
\langle2,5,9\rangle{\upharpoonright}\hspace{0.7pt}\,0=\langle\rangle;
$$
$$
\langle3\rangle^{4}=\langle3,3,3,3\rangle;\qquad
\langle3\rangle^{0}=\langle\rangle;\qquad
\langle3,3\rangle\hat{\ }\langle3\rangle^\omega=\langle3\rangle^\omega;
$$
$$
{}^{1}\hspace{-1pt}{4}=
\big\{\langle0\rangle,\langle1\rangle,\langle2\rangle,\langle3\rangle\big\};\qquad
{}^{0}4=\big\{\langle\rangle\big\};\qquad
{}^{4}0=\varnothing;
$$
$$
{}^{3}\hspace{-1pt}{1}=\big\{\langle0,0,0\rangle\big\};\qquad
{}^\omega{1}=\big\{\langle0\rangle^\omega\big\};\qquad
{}^{{<}\omega}1=\big\{\langle\rangle,\langle0\rangle,\langle0,0\rangle,\ldots\big\}.
$$

\begin{proof}[\textbf{\textup{Proof of Lemma\,\ref{lem2}}}]
We must build a simple closed curve in $\mathbb{R}^3$ whose range $\Gamma$ satisfies
$\mathbb{C}^3\subseteq\Gamma\subseteq[0,1]^3.$ Recall that a \textit{simple closed curve} 
in a space ${X}$ is a continuous injection from the unit circle $\mathbb{S}^1\hspace{-1pt}\subseteq\mathbb{R}^2$ to ${X}.$ (In some books ``simple closed curve'' means ``a space homeomorphic to $\mathbb{S}^1$''; this is because every simple closed curve $\gamma$ in a Hausdorff space is a homeomorphism between $\mathbb{S}^1$ and the range of $\gamma$.)

First we build two indexed families
$\langle{\inn}_{{s}}\rangle_{{s}\in{}^{{<}\omega}{8}}$ and
$\langle{\outt}_{{s}}\rangle_{{s}\in{}^{{<}\omega}{8}}$ of points in $\mathbb{R},$
an indexed family
$\langle{K}_{{s}}\rangle_{{s}\in{}^{{<}\omega}{8}}$ of segments in $\mathbb{R},$
and an indexed family
$\langle{I}_{{s}}({j})\rangle_{\langle{s},{j}\rangle\in{}^{{<}\omega}{8}\times{7}}$ of open intervals in $\mathbb{R}$
by recursion on $\lh({s})\in\omega:$

\begin{itemize}
\item[\textup{(a1)}]
$\inn_{\langle\rangle}\coloneq{0}\in\mathbb{R},$
   $\outt_{\langle\rangle}\coloneq{1}\in\mathbb{R},$
   ${K}_{\langle\rangle}\coloneq
   [\inn_{\langle\rangle},\outt_{\langle\rangle}]=
   [0,1]\subseteq\mathbb{R};$
\item[\textup{(a2)}]
   $\inn_{{s}^\frown{i}}\coloneq
   \inn_{{s}}+\frac{2{i}}{15}(\outt_{{s}}-\inn_{{s}})\in{K}_{{s}}\quad$
   for all ${s}\in{}^{{<}\omega}{8}$ and ${i}\in{8};$
\item[\textup{(a3)}]
   $\outt_{{s}^\frown{i}}\coloneq
   \inn_{{s}^\frown{i}}+\frac{1}{15}(\outt_{{s}}-\inn_{{s}})\in{K}_{{s}}\quad$
   for all ${s}\in{}^{{<}\omega}{8}$ and ${i}\in{8};$
\item[\textup{(a4)}]
   ${K}_{{s}^\frown{i}}\coloneq
   [\inn_{{s}^\frown{i}},\outt_{{s}^\frown{i}}]
   \subseteq{K}_{{s}}\quad$
   for all ${s}\in{}^{{<}\omega}{8}$  and ${i}\in{8};$
\item[\textup{(a5)}]
   ${I}_{{s}}({j})\coloneq
   (\outt_{{s}^\frown{j}},\inn_{{s}^\frown({j}{+}{1})})
   \subseteq{K}_{{s}}\quad$
   for all ${s}\in{}^{{<}\omega}{8}$ and ${j}\in{7}.$
\end{itemize}
So, for each ${s}\in{}^{{<}\omega}{8},$ we have 16 consecutive points $\inn_{{s}}=\inn_{{s}^\frown{0}}
<\outt_{{s}^\frown{0}}<\ldots<\inn_{{s}^\frown{7}}
<\outt_{{s}^\frown{7}}=\outt_{{s}},$ which
cut the segment ${K}_{{s}}=[\inn_{{s}},\outt_{{s}}]$
into 15 consecutive intervals
${K}_{{s}^\frown{0}},{I}_{{s}}({0}),{K}_{{s}^\frown{1}},\ldots,
{I}_{{s}}({6}),{K}_{{s}^\frown{7}}$ of length
$\frac{1}{15}(\outt_{{s}}-\inn_{{s}})$ with
${K}_{{s}^\frown{i}}$ closed intervals and ${I}_{{s}}({j})$ open intervals.
\medskip

We shall use the following properties of these points and intervals:

\begin{itemize}
\item[\textup{(1)}]
   ${\inn}_{{s}}={\inn}_{{s}\hat{\ }\langle{0}\rangle^{n}}$ and
   ${\outt}_{{s}}={\outt}_{{s}\hat{\ }\langle{7}\rangle^{n}}\quad$
   for all ${s}\in{}^{{<}\omega}{8}$ and ${n}\in\omega;$
\item[\textup{(2)}]
   ${K}_{{s}}\:=\;
   \bigsqcup_{{j}\in{7}}{I}_{{s}}({j})\:\sqcup\;
   \bigsqcup_{{i}\in{8}}{K}_{{s}^\frown{i}}\quad$
   for all ${s}\in{}^{{<}\omega}{8}.$
\end{itemize}
\medskip

For each ${x}\in{}^{\omega}{8},$ the set $\bigcap_{{n}\in\omega}{K}_{{x}{\upharpoonright}\hspace{0.7pt}{n}}$
is a singleton, and we define

\begin{itemize}
\item[\textup{(a6)}]
   ${k}_{{x}}\coloneq$ the point in $\mathbb{R}$ such that $\{{k}_{{x}}\}=\bigcap_{{n}\in\omega}{K}_{{x}{\upharpoonright}\hspace{0.7pt}{n}}\quad$
   for all ${x}\in{}^{\omega}{8}.$
\end{itemize}
It follows from (a1) and (a4) that ${\inn}_{{t}},{\outt}_{{t}}\in{K}_{{t}}$ for all ${t}\in{}^{{<}\omega}{8},$
so using (1) we have
${\inn}_{{s}}={\inn}_{{s}\hat{\ }\langle{0}\rangle^{{n}}}\in
{K}_{{s}\hat{\ }\langle{0}\rangle^{{n}}}$ and
${\outt}_{{s}}={\outt}_{{s}\hat{\ }\langle{0}\rangle^{{n}}}\in
{K}_{{s}\hat{\ }\langle{0}\rangle^{{n}}}$ for all ${s}\in{}^{{<}\omega}{8}$ and ${n}\in\omega,$ therefore

\begin{itemize}
\item[\textup{(3)}]
   ${\inn}_{{s}}={k}_{{s}\hat{\ }\langle{0}\rangle^\omega}$ and
   ${\outt}_{{s}}={k}_{{s}\hat{\ }\langle{7}\rangle^\omega}\quad$
   for all ${s}\in{}^{{<}\omega}{8}.$
\end{itemize}
Also it follows from (2) that
\begin{itemize}
\item[\textup{(4)}]
   ${K}_{{s}}\:=\;
   \bigsqcup_{\langle{t},{j}\rangle\in{}^{{<}\omega}{8}\times{7}}
   {I}_{{s}\hat{\ }{t}}({j})\:\sqcup\:
   \bigsqcup_{{x}\in{}^{\omega}{8}}\{{k}_{{s}\hat{\ }{x}}\}\quad$
   for all ${s}\in{}^{<\omega}{8}.$
\end{itemize}
\medskip
%

Next we build two indexed families
$\langle\tilde{\inn}_{{s}}\rangle_{{s}\in{}^{{<}\omega}{8}}$ and
$\langle\tilde{\outt}_{{s}}\rangle_{{s}\in{}^{{<}\omega}{8}}$ of points in $\mathbb{R}^3,$
an indexed family
$\langle\tilde{K}_{{s}}\rangle_{{s}\in{}^{{<}\omega}{8}}$ of closed cubs in $\mathbb{R}^3,$
and an indexed family
$\langle\tilde{I}_{{s}}({j})\rangle_{\langle{s},{j}\rangle\in{}^{{<}\omega}{8}\times{7}}$
of 1-dimensional open intervals in $\mathbb{R}^3$
by recursion on $\lh({s})\in\omega:$

\begin{itemize}
\item[\textup{(b1)}]
   $\mathsurround=0pt\tilde{K}_{\langle\rangle}\coloneq[0,1]^{3}\subseteq\mathbb{R}^{3},$
   $\tilde{\inn}_{\langle\rangle}\coloneq\langle{\frac{1}{3}},0,0\rangle\in\tilde{K}_{\langle\rangle},$
   $\tilde{\outt}_{\langle\rangle}\coloneq\langle{\frac{2}{3}},0,0\rangle\in\tilde{K}_{\langle\rangle};$
\item[\textup{(b2)}]
   $\mathsurround=0pt\tilde{K}_{{s}^\frown{0}},\ldots,
   \tilde{K}_{{s}^\frown{7}}\subseteq\tilde{K}_{{s}}$ are the 8 corner closed cubs among 27 pairwise congruent cubs that we get when we cut the cube $\tilde{K}_{{s}}$ by three pairs of planes parallel to faces of $\tilde{K}_{{s}};$

   that is, if $\tilde{K}_{{s}}=\prod_{{l}\in{3}}[{a}_{{l}},{a}_{{l}}+\delta],$ then

   $\mathsurround=0pt\{\tilde{K}_{{s}^\frown{i}}:{i}\in{8}\}\:=\:
   \big\{\prod_{{l}\in{3}}
   [{a}_{{l}}+\frac{2\delta}{3}{h}_{{l}},
   {a}_{{l}}+\frac{2\delta}{3}{h}_{{l}}+\frac{\delta}{3}]\::\:
   {h}_{{0}},{h}_{{1}},{h}_{{2}}\in\{0,1\}\big\};$
\item[\textup{(b3)}]
   $\tilde{\inn}_{{s}^\frown{i}},\tilde{\outt}_{{s}^\frown{i}}\in\tilde{K}_{{s}}$
   are two different vertices of the cube $\tilde{K}_{{s}^\frown{i}}\quad$
   for all ${s}\in{}^{{<}\omega}{8}$ and ${i}\in{8};$
\item[\textup{(b4)}]
   $\tilde{I}_{{s}}({j})\coloneq
   ({\tilde{\outt}}_{{s}^\frown{j}},{\tilde{\inn}}_{{s}^\frown({j}{+}{1})})\subseteq\tilde{K}_{{s}}\quad$
   for all ${s}\in{}^{{<}\omega}{8}$ and ${j}\in{7},$

   where $({a},{b})\coloneq\mathsf{co}(\{{a},{b}\})\setminus\{{a},{b}\}$ is an open interval in $\mathbb{R}^3$ for ${a},{b}\in\mathbb{R}^3.$
\end{itemize}
At each step in this recursion we can enumerate the 8 cubs $\tilde{K}_{{s}^\frown{0}},\ldots,\tilde{K}_{{s}^\frown{7}}$ in such order and choose 8 pairs of their vertices $\tilde{\inn}_{{s}^\frown{i}},\tilde{\outt}_{{s}^\frown{i}}\in\tilde{K}_{{s}^\frown{i}}$ in such a way that:

\begin{itemize}
\item[\ding{226}\,]
   $\tilde{\inn}_{{s}^\frown{0}}=\tilde{\inn}_{{s}}$ and
   $\tilde{\outt}_{{s}^\frown{7}}=\tilde{\outt}_{{s}}\quad$
   for all ${s}\in{}^{{<}\omega}{8};$
\item[\ding{226}\,]
   $\tilde{I}_{{s}}({j})\cap\tilde{K}_{{s}^\frown{i}}=\varnothing\quad$
   for all ${s}\in{}^{{<}\omega}{8},$ ${j}\in{7},$ and ${i}\in{8};$
\item[\ding{226}\,]
   $\tilde{I}_{\langle\rangle}({j})\cap
   (\tilde{\outt}_{\langle\rangle},\tilde{\inn}_{\langle\rangle})=\varnothing\quad$
   for all ${j}\in{7};$
\item[\ding{226}\,]
   $\tilde{I}_{{s}}({j})\cap\tilde{I}_{{s}}({l})=\varnothing\quad$
   for all ${s}\in{}^{{<}\omega}{8}$ and ${j}\neq{l}\in{7}.$
\end{itemize}
(Alternatively, we could take $\tilde{I}_{{s}}({j})$ to be polygonal chains or arcs without endpoints to simplify an argument here; but this would complicate clause (c1) below.)

\medskip

We shall use the following properties of these cubs, intervals, and vertices:

\begin{itemize}
\item[\textup{(5)}]
   $\tilde{\inn}_{{s}}=\tilde{\inn}_{{s}\hat{\ }\langle{0}\rangle^{n}}$ and
   $\tilde{\outt}_{{s}}=\tilde{\outt}_{{s}\hat{\ }\langle{7}\rangle^{n}}\quad$
   for all ${s}\in{}^{{<}\omega}{8}$ and ${n}\in\omega;$
\item[\textup{(6)}]
   $\tilde{K}_{{s}^\frown{i}}\cap\tilde{K}_{{s}^\frown{l}}=\varnothing\quad$
   for all ${s}\in{}^{{<}\omega}{8}$ and $ {i}\neq{l}\in{8};$
\item[\textup{(7)}]
   $\tilde{I}_{{s}\hat{\ }{t}}({j})\subseteq\tilde{K}_{{s}}\quad$
   for all ${s},{t}\in{}^{{<}\omega}{8}$ and ${j}\in{7};$
\item[\textup{(8)}]
   $\tilde{I}_{{s}}({j})\cap\tilde{K}_{{t}}=\varnothing\quad$
   for all ${s},{t}\in{}^{{<}\omega}{8}$ and ${j}\in{7}$
   such that $\lh({t})>\lh({s});$
\item[\textup{(9)}]
   $\tilde{I}_{{s}}({j})\cap\tilde{I}_{{t}}({l})=\varnothing\quad$
   for all $\langle{s},{j}\rangle\neq\langle{t},{l}\rangle\in{}^{{<}\omega}{8}\times{7};$
\item[\textup{(10)}]
   $(\tilde{\outt}_{\langle\rangle},\tilde{\inn}_{\langle\rangle})
   \cap\tilde{K}_{{s}}=\varnothing\quad$
   for all ${s}\in{}^{{<}\omega}{8}$ such that ${s}\neq\langle\rangle;$
\item[\textup{(11)}]
   $(\tilde{\outt}_{\langle\rangle},\tilde{\inn}_{\langle\rangle})
   \cap\tilde{I}_{{s}}({j})=\varnothing\quad$
   for all ${s}\in{}^{{<}\omega}{8}$ and ${j}\in{7}.$
\end{itemize}
\medskip

For each ${x}\in{}^{\omega}{8},$ the set
$\bigcap_{{n}\in\omega}\tilde{K}_{{x}{\upharpoonright}\hspace{0.7pt}{n}}$ is a singleton, and we define

\begin{itemize}
\item[\textup{(b5)}]
   $\tilde{k}_{{x}}\coloneq$ the point in $\mathbb{R}^3$ such that $\{\tilde{k}_{{x}}\}=\bigcap_{{n}\in\omega}\tilde{K}_{{x}{\upharpoonright}\hspace{0.7pt}{n}}\quad$
   for all ${x}\in{}^{\omega}{8}.$
\end{itemize}
Since $\tilde{\inn}_{{t}},\tilde{\outt}_{{t}}\in\tilde{K}_{{t}}$ for all ${t}\in{}^{{<}\omega}{8}$ and since by (5) we have
$\tilde{\inn}_{{s}}=\tilde{\inn}_{{s}\hat{\ }\langle{0}\rangle^{{n}}}\in
\tilde{K}_{{s}\hat{\ }\langle{0}\rangle^{{n}}}$ and
$\tilde{\outt}_{{s}}=\tilde{\outt}_{{s}\hat{\ }\langle{0}\rangle^{{n}}}\in
\tilde{K}_{{s}\hat{\ }\langle{0}\rangle^{{n}}}$ for all ${s}\in{}^{{<}\omega}{8}$ and ${n}\in\omega,$ it follows that

\begin{itemize}
\item[\textup{(12)}]
   $\tilde{\inn}_{{s}}=\tilde{k}_{{s}\hat{\ }\langle{0}\rangle^\omega}$ and
   $\tilde{\outt}_{{s}}=\tilde{k}_{{s}\hat{\ }\langle{7}\rangle^\omega}\quad$
   for all ${s}\in{}^{{<}\omega}{8}.$
\end{itemize}
Also we have

\begin{itemize}
\item[\textup{(13)}]
   $\mathbb{C}^{3}=\{\tilde{k}_{{x}}:{x}\in{}^\omega{8}\},$ where $\mathbb{C}$ is the Cantor set.
\end{itemize}
\medskip

Now we build a mapping $\gamma$ from $[0,2\pi)\subseteq\mathbb{R}$ to $\mathbb{R}^3.$
Property (4) with ${s}=\langle\rangle$ says that

\begin{itemize}
\item[\textup{(14)}]
   $[0,1]\:=\;
   \bigsqcup_{\langle{t},{j}\rangle\in{}^{{<}\omega}{8}\times{7}}{I}_{{t}}({j})
   \:\sqcup\:\bigsqcup_{{x}\in{}^{\omega}{8}}\{{k}_{{x}}\},$
\end{itemize}
so we may specify the values of $\gamma$ independently on points of the interval $(1,2\pi)$ and on points of different members of disjoint union in (14).
For ${a}\neq{b}\in\mathbb{R}$ and ${c}\neq{d}\in\mathbb{R}^3,$ let $\mathsf{LIN}^{{a},{b}}_{{c},{d}}$ denote the linear mapping ${l}\colon[{a},{b}]\to[{c},{d}]\subseteq\mathbb{R}^3$ such that ${l}({a})={c}$ and ${l}({b})={d}.$
Recall that
${I}_{{s}}({j})=(\outt_{{s}^\frown{j}},\inn_{{s}^\frown({j}{+}{1})})$ and
$\tilde{I}_{{s}}({j})=({\tilde{\outt}}_{{s}^\frown{j}},{\tilde{\inn}}_{{s}^\frown({j}{+}{1})}).$ Let $\gamma$ be a function such that:

\begin{itemize}
\item[\textup{(c1)}]
   $\gamma{\upharpoonright}\hspace{0.7pt}{I}_{{s}}({j})=
   \mathsf{LIN}^{\outt_{{s}^\frown{j}},\inn_{{s}^\frown({j}{+}{1})}}
   _{{\tilde{\outt}}_{{s}^\frown{j}},{\tilde{\inn}}_{{s}^\frown({j}{+}{1})}}
   {\upharpoonright}\hspace{0.7pt}{I}_{{s}}({j})\quad$
   for all ${s}\in{}^{{<}\omega}{8}$ and $ {j}\in{7};$
\item[\textup{(c2)}]
   $\gamma{\upharpoonright}\hspace{0.7pt}({1},2\pi)=
   \mathsf{LIN}^{{1},2\pi}
   _{\tilde{\outt}_{\langle\rangle},\tilde{\inn}_{\langle\rangle}}
   {\upharpoonright}\hspace{0.7pt}({1},2\pi);$
\item[\textup{(c3)}]
   $\gamma({k}_{{x}})=\tilde{k}_{{x}}\quad$
   for all ${x}\in{}^{\omega}{8}.$
\end{itemize}
\medskip

The mapping $\gamma\colon[0,2\pi)\to\mathbb{R}^3$ has the following properties:

\begin{itemize}
\item[\textup{(15)}]
   $\gamma\big({I}_{{s}}({j})\big)=\tilde{I}_{{s}}({j})$
   and $\gamma{\upharpoonright}\hspace{0.7pt}{I}_{{s}}({j})$ is injection
   for all ${s}\in{}^{{<}\omega}{8}$ and $ {j}\in{7};$
\item[\textup{(16)}]
   $\gamma\big(({1},2\pi)\big)=
   (\tilde{\outt}_{\langle\rangle},\tilde{\inn}_{\langle\rangle})$ and $\gamma{\upharpoonright}\hspace{0.7pt}({1},2\pi)$ is injection;
\item[\textup{(17)}]
   $\gamma({K}_{{s}})\subseteq\tilde{K}_{{s}}\quad$
   for all ${s}\in{}^{{<}\omega}{8}$

   (this follows from (4), (15), (7), (c3) and from the fact that $\tilde{k}_{{s}\hat{\ }{x}}\in\tilde{K}_{{s}}$ for all ${x}\in{}^\omega{8}$);
\item[\textup{(18)}]
   $\gamma(\inn_{{s}})=\tilde{\inn}_{{s}}$ and
   $\gamma(\outt_{{s}})=\tilde{\outt}_{{s}}\quad$
   for all ${s}\in{}^{{<}\omega}{8}$

   (this follows from (3) and (12));
\item[\textup{(19)}]
   $\textsf{domain}(\gamma)\:=\:[0,2\pi)\:=\:
   ({1},2\pi)\:\sqcup\:
   \bigsqcup_{\langle{s},{j}\rangle\in{}^{{<}\omega}{8}\times{7}}{I}_{{s}}({j})\:\sqcup\:
   \bigsqcup_{{x}\in{}^{\omega}{8}}\{{k}_{{x}}\}$

   (this follows from (14));
\item[\textup{(20)}]
   $\textsf{range}(\gamma)\:=\:
   (\tilde{\outt}_{\langle\rangle},\tilde{\inn}_{\langle\rangle})\:\sqcup\:
   \bigsqcup_{\langle{s},{j}\rangle\in{}^{{<}\omega}{8}\times{7}}\tilde{I}_{{s}}({j})\:\sqcup\:
   \bigsqcup_{{x}\in{}^{\omega}{8}}\{\tilde{k}_{{x}}\}$

   (equality
   $\textsf{range}(\gamma)=
   (\tilde{\outt}_{\langle\rangle},\tilde{\inn}_{\langle\rangle})\cup
   \bigcup_{\langle{s},{j}\rangle\in{}^{{<}\omega}{8}\times{7}}\tilde{I}_{{s}}({j})\cup
   \bigcup_{{x}\in{}^{\omega}{8}}\{\tilde{k}_{{x}}\}$
   follows from (19), (16), (15), and (c3);
   $\{\tilde{k}_{{x}}\}\cap\{\tilde{k}_{{y}}\}=\varnothing$ for ${x}\neq{y}\in{}^{\omega}{8}$
   follows from (6);
   $\tilde{I}_{{s}}({j})\cap\tilde{I}_{{t}}({l})=\varnothing$
   for $\langle{s},{j}\rangle\neq\langle{t},{l}\rangle\in{}^{{<}\omega}{8}\times{7}$
   is asserted in (9);
   $\tilde{I}_{{s}}({j})\cap\{\tilde{k}_{{x}}\}=\varnothing$
   follows from (8);
   $(\tilde{\outt}_{\langle\rangle},\tilde{\inn}_{\langle\rangle})
   \cap\{\tilde{k}_{{x}}\}=\varnothing$
   follows from (10); and
   $(\tilde{\outt}_{\langle\rangle},\tilde{\inn}_{\langle\rangle})
   \cap\tilde{I}_{{s}}({j})=\varnothing$ is asserted in (11)).
\end{itemize}
\medskip

Let $\mathbb{S}$ be the half-interval $[0,2\pi)\subseteq\mathbb{R}$ endowed with the metric $\rho_{\scriptscriptstyle\mathbb{S}},$ which is defined as follows:
$\rho_{\scriptscriptstyle\mathbb{S}}({x},{y})\coloneq\mathsf{min}\{|{x}-{y}|,2\pi-|{x}-{y}|\}.$
The metric space $\mathbb{S}$ is homeomorphic to the unit circle $\mathbb{S}^1\coloneq\{\langle{x}_{1},{x}_{2}\rangle\in\mathbb{R}^{2}:{x}_{1}^{2}+{x}_{2}^{2}=1\}$
by the homeomorphism ${h}:\mathbb{S}\to\mathbb{S}^1$ that takes a point ${x}$ of $\mathbb{S}$ to the point ${h}({x})$ of $\mathbb{S}^1$ such that the polar angle of ${h}({x})$ equals ${x}$ radians.
We shall show that
\begin{itemize}
\item [\ding{226}\,]
   $\mathsurround=0pt
   \gamma\colon\mathbb{S}\to\mathbb{R}^3$ is a continuous injection and
\item [\ding{226}\,]
   $\mathsurround=0pt
   \mathbb{C}^{3}\subseteq\Gamma\subseteq[0,1]^{3},$
\end{itemize}
where $\Gamma\coloneq\textsf{range}(\gamma),$ so that
the composition $\gamma\circ{h}^{{-}1}:\mathbb{S}^1\to\mathbb{R}^3$ is a simple closed curve that we search.

The inclusion $\mathbb{C}^{3}\subseteq\Gamma$ follows from (13), which says that $\mathbb{C}^{3}=\{\tilde{k}_{{x}}:{x}\in{}^\omega{8}\},$ and from (20), which implies $\{\tilde{k}_{{x}}:{x}\in{}^\omega{8}\}\subseteq\Gamma.$ The inclusion $\Gamma\subseteq[0,1]^{3}$ follows from (20) and (b1)--(b5).
Also (19), (16), (15), (c3), and (20) imply that $\gamma$ is an injection.

It remains to show that $\gamma$ is continuous at every point ${p}$ of $\mathbb{S}.$ By (19) we have three cases:

\begin{itemize}
\item[\textup{(i)}]
   ${p}\in({1},2\pi).$

   Then $\gamma$ is continuous at ${p}$ because $({1},2\pi)$ is open in $\mathbb{S}.$
\item[\textup{(ii)}]
   ${p}\in{I}_{{s}}({j})\quad$
   for some ${s}\in{}^{{<}\omega}{8}$ and ${j}\in{7}.$

   Then $\gamma$ is continuous at ${p}$ because ${I}_{{s}}({j})$ is open in $\mathbb{S}.$
\item[\textup{(iii)}]
   ${p}={k}_{{x}}$ for some ${x}\in{}^\omega{8}.$

   Let $\varepsilon >0.$ We shall find $\delta >0$ such that $\mathsf{diam}\Big(\gamma\big(O_{\delta}({k}_{{x}})\big)\Big)<\varepsilon,$
   where $\mathsf{diam}({A})$ is the diameter of a set ${A}$ and $O_{\delta}({q})$ is the $\delta\mathsurround=0pt$-neighbourhood of a point ${q}$ in $\mathbb{S}.$
   Let ${m}\in\omega$ be such that
   $\mathsf{diam}(\tilde{K}_{{x}{\upharpoonright}\hspace{0.7pt}{m}})<\varepsilon/2.$
   Recall that
   ${k}_{{x}}\in{K}_{{x}{\upharpoonright}\hspace{0.7pt}{m}}=
   [\inn_{{x}{\upharpoonright}\hspace{0.7pt}{m}},\outt_{{x}{\upharpoonright}\hspace{0.7pt}{m}}]$ and
   $\gamma({k}_{{x}})=\tilde{k}_{{x}}\in\tilde{K}_{{x}{\upharpoonright}\hspace{0.7pt}{m}}.$
   \begin{itemize}
   \item[\textup{(iii.1)}]
      ${k}_{{x}}\in(\inn_{{x}{\upharpoonright}\hspace{0.7pt}{m}},\outt_{{x}{\upharpoonright}\hspace{0.7pt}{m}}).$

      Then there is $\delta >0$ such that $O_{\delta}({k}_{{x}})\subseteq{K}_{{x}{\upharpoonright}\hspace{0.7pt}{m}},$
      so $\gamma\big(O_{\delta}({k}_{{x}})\big)\subseteq
      \gamma({K}_{{x}{\upharpoonright}\hspace{0.7pt}{m}})\subseteq
      \tilde{K}_{{x}{\upharpoonright}\hspace{0.7pt}{m}}$ by (17), whence $\mathsf{diam}\Big(\gamma\big(O_{\delta}({k}_{{x}})\big)\Big)<\varepsilon/2<\varepsilon$
      by the choice of ${m}.$
   \item[\textup{(iii.2)}]
      ${k}_{{x}}=\inn_{{x}{\upharpoonright}\hspace{0.7pt}{m}}.$
      \begin{itemize}
      \item[\textup{(iii.2.1)}]
         ${x}{\upharpoonright}\hspace{0.7pt}{m}=\langle{0}\rangle^{{m}}.$

         By (1), ${\inn}_{\langle{0}\rangle^{{m}}}={\inn}_{\langle\rangle},$
         so we have ${k}_{{x}}=\inn_{{x}{\upharpoonright}\hspace{0.7pt}{m}}=
         \inn_{\langle\rangle}={0}.$
         Recall that
         $\gamma{\upharpoonright}\hspace{0.7pt}(1,2\pi)=
         \mathsf{LIN}^{1,2\pi}   _{\tilde\outt_{\langle\rangle},
         \tilde\inn_{\langle\rangle}}{\upharpoonright}\hspace{0.7pt}(1,2\pi)$
         by (c2). There is $\delta_{1}>0$ such that
         $\mathsf{diam}\Big(\gamma\big((2\pi-\delta_{1},2\pi)\big)\Big)<\varepsilon/2.$
         Then also
         $\mathsf{diam}\Big(\gamma\big((2\pi-\delta_{1},2\pi)\cup\{0\}\big)\Big)<\varepsilon/2$
         because $\gamma({0})=\gamma({\inn}_{\langle\rangle})=\tilde{\inn}_{\langle\rangle}$
         by (18) and $\tilde{\inn}_{\langle\rangle}$ is a limit point of
         $\gamma\big((2\pi-\delta_{1},2\pi)\big).$
         Since ${K}_{{x}{\upharpoonright}\hspace{0.7pt}{m}}=
         [\inn_{{x}{\upharpoonright}\hspace{0.7pt}{m}},\outt_{{x}{\upharpoonright}\hspace{0.7pt}{m}}],$
         there is $\delta_2>0$ such that
         $[0,\delta_{2})=
         [\inn_{{x}{\upharpoonright}\hspace{0.7pt}{m}},\inn_{{x}{\upharpoonright}\hspace{0.7pt}{m}}+\delta_{2})
         \subseteq{K}_{{x}{\upharpoonright}\hspace{0.7pt}{m}}.$
         Then
         $\gamma\big([{0},\delta_{2})\big)
         \subseteq\gamma({K}_{{x}{\upharpoonright}\hspace{0.7pt}{m}})\subseteq
         \tilde{K}_{{x}{\upharpoonright}\hspace{0.7pt}{m}},$
         hence $\mathsf{diam}\Big(\gamma\big([0,\delta_{2})\big)\Big)
         \leqslant\mathsf{diam}(\tilde{K}_{{x}{\upharpoonright}\hspace{0.7pt}{m}})<\varepsilon/2.$
         So if $\delta=\mathsf{min}\{\delta_{1},\delta_{2}\},$ then
         using the fact that $\mathsf{diam}({A}\cup{B})\leqslant
         \mathsf{diam}({A})+\mathsf{diam}({B})$ whenever ${A}\cap{B}\neq\varnothing,$
         we have
         $$
         \mathsf{diam}\Big(\gamma\big(O_{\delta}({k}_{{x}})\big)\Big)=
         \mathsf{diam}\Big(\gamma\big(O_{\delta}(0)\big)\Big)\leqslant
         \mathsf{diam}\Big(\gamma\big((2\pi-\delta_{1},2\pi)\cup[0,\delta_{2})\big)\Big)
         \leqslant
         $$
         $$
         \mathsf{diam}\Big(\gamma\big((2\pi-\delta_{1},2\pi)\cup\{0\}\big)\Big)+
         \mathsf{diam}\Big(\gamma\big([0,\delta_{2})\big)\Big)<\varepsilon/2+\varepsilon/2
         =\varepsilon.
         $$
      \item[\textup{(iii.2.2)}]
         ${x}{\upharpoonright}\hspace{0.7pt}{m}\neq\langle{0}\rangle^{{m}}.$

         Then ${m}>0$ because ${x}{\upharpoonright}\hspace{0.7pt}0=\langle\rangle=\langle0\rangle^{0}.$
         Let ${i}\coloneq\mathsf{max}\{{j}\in\omega:{j}<{m}\text{ and }{x}_{{j}}>0\},$
         so that ${x}_{{i}}>0$ and
         $$
         {x}{\upharpoonright}\hspace{0.7pt}{m}=\langle{x}_{{0}},\ldots,{x}_{{i}{-}{1}}\rangle
         ^\frown{x}_{{i}}\hat{\;}\langle{0}\rangle^{{m}{-}{i}{-}{1}}=
         ({x}{\upharpoonright}\hspace{0.7pt}{i})^\frown{x}_{{i}}\hat{\;}
         \langle{0}\rangle^{{m}{-}{i}{-}{1}}.
         $$
         Then using (iii.2) and (1) we have
         ${k}_{{x}}=\inn_{{x}{\upharpoonright}\hspace{0.7pt}{m}}=
         \inn_{({x}{\upharpoonright}\hspace{0.7pt}{i})^\frown{x}_{{i}}}.$
         Recall that
         ${I}_{{x}{\upharpoonright}\hspace{0.7pt}{i}}({x}_{{i}}{-}{1})=
         (\outt_{({x}{\upharpoonright}\hspace{0.7pt}{i})^\frown({x}_{{i}}-{1})},
         \inn_{({x}{\upharpoonright}\hspace{0.7pt}{i})^\frown{x}_{{i}}})$ by (a5), that
         $$
         \gamma{\upharpoonright}\hspace{0.7pt}{I}_{{x}{\upharpoonright}\hspace{0.7pt}{i}}({x}_{{i}}{-}{1})=   \mathsf{LIN}^{\outt_{({x}{\upharpoonright}\hspace{0.7pt}{i})^\frown({x}_{{i}}-{1})},
         \inn_{({x}{\upharpoonright}\hspace{0.7pt}{i})^\frown{x}_{{i}}}}   _{\tilde\outt_{({x}{\upharpoonright}\hspace{0.7pt}{i})^\frown({x}_{{i}}-{1})},
         \tilde\inn_{({x}{\upharpoonright}\hspace{0.7pt}{i})^\frown{x}_{{i}}}}   {\upharpoonright}\hspace{0.7pt}{I}_{{x}{\upharpoonright}\hspace{0.7pt}{i}}({x}_{{i}}{-}{1})
         $$
         by (c1), and that
         $\gamma(\inn_{({x}{\upharpoonright}\hspace{0.7pt}{i})^\frown{x}_{{i}}})= \tilde{\inn}_{({x}{\upharpoonright}\hspace{0.7pt}{i})^\frown{x}_{{i}}}$
         by (18). So there is $\delta_{1}>0$ such that\linebreak
         $
         \mathsf{diam}\Big(\gamma\big((\inn_{({x}{\upharpoonright}
         \hspace{0.7pt}{i})^\frown{x}_{{i}}}-\,\delta_1,
         \inn_{({x}{\upharpoonright}\hspace{0.7pt}{i})^\frown{x}_{{i}}}]\big)\Big)<\varepsilon/2.
         $
         That is,
         $\mathsf{diam}\Big(\gamma\big(({k}_{{x}}-\delta_1,{k}_{{x}}]\big)\Big)<\varepsilon/2.$
         Also we have
         $[{k}_{{x}},\outt_{{x}{\upharpoonright}\hspace{0.7pt}{m}}]=
         [\inn_{{x}{\upharpoonright}\hspace{0.7pt}{m}},\outt_{{x}{\upharpoonright}\hspace{0.7pt}{m}}]=
         {K}_{{x}{\upharpoonright}\hspace{0.7pt}{m}},$ hence there is $\delta_{2}>0$ such that
         $[{k}_{{x}},{k}_{{x}}+\delta_2)\subseteq
         {K}_{{x}{\upharpoonright}\hspace{0.7pt}{m}},$
         and then $\mathsf{diam}\Big(\gamma\big([{k}_{{x}},{k}_{{x}}+\delta_2)\big)\Big)<\varepsilon/2$
         by the choice of ${m}.$
         So if $\delta=\mathsf{min}\{\delta_1,\delta_2\},$ then
         $$
         \mathsf{diam}\Big(\gamma\big(O_{\delta}({k}_{{x}})\big)\Big)\leqslant
         \mathsf{diam}\Big(\gamma\big(({k}_{{x}}-\delta_1,{k}_{{x}}]\big)\Big)+
         \mathsf{diam}\Big(\gamma\big([{k}_{{x}},{k}_{{x}}+\delta_2)\big)\Big)<\varepsilon.
         $$
      \end{itemize}
   \item[\textup{(iii.3)}]
      ${k}_{{x}}=\outt_{{x}{\upharpoonright}\hspace{0.7pt}{m}}.$

      This case is similar to (iii.2).
   \end{itemize}
\end{itemize}
\end{proof}


\end{document}